\documentclass[11pt,a4paper]{article}
\usepackage{fullpage}

\usepackage[utf8]{inputenc}
\usepackage{amsthm,amssymb}
\usepackage[leqno]{amsmath}
\usepackage{tcolorbox}
\usepackage{thmtools}
\usepackage{mathtools}
\usepackage{subcaption,thm-restate}
\usepackage[mathcal,mathscr]{eucal}
\usepackage{epsfig,graphicx,graphics,color}
\usepackage{caption}
\usepackage{enumerate}
\usepackage{enumitem, crossreftools}
\usepackage[sort]{cite}
\usepackage{hyperref}
\usepackage{tikz}
\hypersetup{
    colorlinks=true,       
    linkcolor=blue,        
    citecolor=red,         
    filecolor=magenta,     
    urlcolor=cyan,         
    linktocpage=true
}

\theoremstyle{plain}
\newtheorem{thm}{Theorem}
\newtheorem{lem}[thm]{Lemma}

\newtheorem{conj}{Conjecture}

\theoremstyle{definition}

\newtheorem{rem}[thm]{Remark}

\def\final{0}  
\def\iflong{\iffalse}
\ifnum\final=0  
\newcommand{\knote}[1]{{\color{red}[{\tiny \textbf{Kristóf:} \bf #1}]\marginpar{\color{red}*}}}
\newcommand{\tnote}[1]{{\color{blue}[{\tiny \textbf{Tamás:} \bf #1}]\marginpar{\color{blue}*}}}
\newcommand{\bnote}[1]{{\color{green}[{\tiny \textbf{Bence:} \bf #1}]\marginpar{\color{green}*}}}
\else 
\newcommand{\knote}[1]{}
\newcommand{\tnote}[1]{}
\newcommand{\bnote}[1]{}
\fi

\DeclareMathOperator\ineqa{A}
\DeclareMathOperator\ineqb{B}

\newcommand{\bP}{\mathbf{P}}

\newcommand{\cB}{\mathcal{B}}
\newcommand{\cI}{\mathcal{I}}

\newcommand{\cH}{\mathcal{H}}

\newcommand{\cC}{\mathcal{C}}

\title{Weighted exchange distance of basis pairs}
\date{}

\author{
Kristóf Bérczi\thanks{MTA-ELTE Momentum Matroid Optimization Research Group and MTA-ELTE Egerváry Research Group, Department of Operations Research, Eötvös Loránd University, Budapest, Hungary. Email: \texttt{kristof.berczi@ttk.elte.hu, matben@student.elte.hu, tamas.schwarcz@ttk.elte.hu}.}
\and
Bence Mátravölgyi\footnotemark[1]
\and
Tamás Schwarcz\footnotemark[1]
}

\begin{document}
\maketitle

\begin{abstract}
Two pairs of disjoint bases $\bP_1=(R_1,B_1)$ and $\bP_2=(R_2,B_2)$ of a matroid $M$ are called \emph{equivalent} if $\bP_1$ can be transformed into $\bP_2$ by a series of symmetric exchanges. In 1980, White conjectured that such a sequence always exists whenever $R_1\cup B_1=R_2\cup B_2$. A strengthening of the conjecture was proposed by Hamidoune, stating that minimum length of an exchange is at most the rank of the matroid.

We propose a weighted variant of Hamidoune's conjecture, where the weight of an exchange depends on the weights of the exchanged elements. We prove the conjecture for several matroid classes: strongly base orderable matroids, split matroids, graphic matroids of wheels, and spikes.
\medskip

\end{abstract}

\section{Introduction}
\label{sec:intro}

Given a matroid $M$ over a ground set $S$, the exchange axiom implies that for any pair of bases $R$ and $B$ there exists a sequence of exchanges that transforms $R$ into $B$, and the shortest length of such a sequence is $|R-B|$. In the light of this, it is natural to ask whether analogous results hold for basis pairs instead of single basis. More precisely, let $(R,B)$ be an ordered pair of disjoint bases of $M$, and let $e\in R\setminus B$ and $f\in B\setminus R$ be such that both $R'\coloneqq R-e+f$ and $B'\coloneqq B+e-f$ are bases. In such a case, we call the exchange \textbf{feasible} and say that the pair $(R',B')$ is obtained from $(R,B)$ by a \textbf{symmetric exchange}. Using this terminology, we define the \textbf{exchange distance} (or \textbf{distance} for short) of two pairs of disjoint bases $\bP_1=(R_1,B_1)$ and $\bP_2=(R_2,B_2)$ to be the minimum number of symmetric exchanges needed to transform the former into the latter if such a sequence exists and $+\infty$ otherwise. We call two pairs of disjoint bases \textbf{equivalent} if their exchange distance is finite. A sequence of symmetric exchanges starting from a pair $\bP_1$ is called \textbf{strictly monotone with respect to another pair $\bP_2$} (or \textbf{strictly monotone} for short when $\bP_2$ is clear from the context) if each step decreases the difference between the first member of the current pair and that of $\bP_2$. In other words, a strictly monotone exchange sequence uses elements only from $(R_1\cap B_2)\cup(R_2\cap B_1)$ and at most once.

At this point it is not clear (I) when the distance of two pairs will be finite, and (II) if their distance is  finite, whether we can give an upper bound on it. Regarding question (I), one can formulate an obvious necessary condition for the distance of $\bP_1$ and $\bP_2$ to be finite, namely $R_1\cup B_1=R_2\cup B_2$ should certainly hold -- two pairs with this property are called \textbf{compatible}. In \cite{white1980unique}, White conjectured that \emph{two basis pairs $\bP_1$ and $\bP_2$ are equivalent if and only if they are compatible}. While the conjecture was verified for various matroid classes, it remains open in general. 

Much less is known about question (II), that is, the optimization version of the problem. Gabow~\cite{gabow1976decomposing} studied sequential symmetric exchanges and posed the following problem, which was later stated as a conjecture by Wiedemann~\cite{wiedemann1984cyclic} and by Cordovil and Moreira~\cite{cordovil1993bases}: \emph{for any two disjoint bases $R$ and $B$ of a matroid $M$, there is a sequence of $r$ symmetric exchanges that transforms the pair $\bP_1=(R,B)$ into $\bP_2=(B,R)$}. The rank of the matroid is a trivial lower bound on the minimum number of exchanges needed to transform a pair $(R,B)$ into $(B,R)$, and the essence of Gabow's conjecture is that that many steps might always suffice. The relation between the conjectures of White and Gabow is immediate: the latter would imply the former for sequences of the form $(R,B)$ and $(B,R)$.

In general, if $M$ has rank $r$, then $r-|R_1\cap R_2|$ is an obvious lower bound on the exchange distance of $\bP_1=(R_1,B_1)$ and $\bP_2 = (R_2,B_2)$. However, it might happen that more symmetric exchanges are needed even if $M$ is a graphic matroid; see \cite{farber1989basis} or Figure~\ref{fig:k4} for a counterexample. As a common generalization of the conjectures of White and Gabow, Hamidoune~\cite{cordovil1993bases} proposed a rather optimistic variant stating that \emph{the exchange distance of compatible basis pairs is at most the rank of the matroid}. 

Let $w\colon S\to\mathbb{R}_+$ be a weight function on the elements of the ground set $S$. Given a pair $(R,B)$ of disjoint bases, we define the \textbf{weight of a symmetric exchange}  $R'\coloneqq R-e+f$ and $B'\coloneqq B+e-f$ to be $w(e)+w(f)$, that is, the sum of the weights of the exchanged elements. Analogously to the unweighted setting, we define the \textbf{weighted exchange distance} (or \textbf{weighted distance} for short) of two pairs of disjoint bases $\bP_1=(R_1,B_1)$ and $\bP_2=(R_2,B_2)$ to be the minimum total weight of symmetric exchanges needed to transform the former into the latter if such a sequence exists and $+\infty$ otherwise. As a weighted extension of Hamidoune's conjecture, we propose the following.

\begin{conj}\label{conj:weighted}
Let $\bP_1=(R_1,B_1)$ and $\bP_2=(R_2,B_2)$ be compatible basis pairs of a matroid $M$ over a ground set $S$, and let $w\colon S\to\mathbb{R}_+$ be a weight function. Then the weighted exchange distance of $\bP_1$ and $\bP_2$ is at most $w(R_1\cup B_1)=w(R_2\cup B_2)$.
\end{conj}

By setting the weight function to be identically one, we get back Hamidoune's conjecture. It is worth mentioning that a strictly monotone exchange sequence transforming $\bP_1$ into $\bP_2$ is optimal in every sense, i.e., it has both minimum length and minimum weight. 

\paragraph{Previous work.}

By relying on the constructive characterization of bispanning graphs, Farber, Richter, and Shank~\cite{farber1985edge} proved White's conjecture for graphic and cographic matroids, while Farber~\cite{farber1989basis} settled the statement for transversal matroids. Bonin~\cite{bonin2013basis} verified the conjecture for sparse paving matroids. The case of strongly base orderable matroids was solved by Laso{\'n} and Micha{\l}ek~\cite{lason2014toric}. McGuinness~\cite{mcguinness2020frame} extended the graphic case to frame matroids satisfying a certain linearity condition. Kotlar and Ziv~\cite{kotlar2013serial} showed that any two elements of a basis have a sequential symmetric exchange with some two elements of any other basis. At the same time, Kotlar~\cite{kotlar2013circuits} proved that three consecutive symmetric exchanges exist for any two bases of a matroid, and that a full sequential symmetric exchange, of length at most $6$, exists for any two bases of a matroid of rank $5$.

Gabow's conjecture was verified for partition matroids, matching and transversal matroids, and matroids of rank less than $4$ in~\cite{gabow1976decomposing}, and an easy reasoning shows that it also holds for strongly base orderable matroids as well. The graphic case was settled by Wiedemann~\cite{wiedemann1984cyclic}, Kajitani, Ueno, and Miyano~\cite{kajitani1988ordering}, and Cordovil and Moreira~\cite{cordovil1993bases}.  

Recently, Bérczi and Schwarcz~\cite{berczi2022exchange} showed that Hamidoune's conjecture hold for split matroids, a large class that contains paving matroids as well. While studying a specific maker-breaker game on bispanning graphs, Andres, Hochst\"{a}ttler and Merkel \cite{andres2014base} showed that there is an exchange sequence between any two pairs of disjoint spanning trees of a wheel of rank at least four using only so-called left unique exchanges. They also asked whether the exchange distance of compatible basis pairs of a matroid can be bounded by a polynomial of the rank -- this latter problem is a weakening of Hamidoune's conjecture. 

The rank of the graphic matroid of a connected graph on $n$ vertices is $n-1$. Though it is not stated explicitly, the algorithms of \cite{farber1985edge} and \cite{blasiak2008toric} that prove White's conjecture for graphic matroids give a sequence of exchanges of length at most $O(n^2)$. It remains an intriguing open problem to improve the bound to $O(n)$, matching the order of the bound in the conjecture.

\paragraph{Our results.}

We verify Conjecture~\ref{conj:weighted} for various matroid classes. First we consider strongly base orderable matroids, a class with distinctive structural properties.

For the remaining matroid classes, we work with a further strengthening of the conjecture where both the length and the weight of the exchange sequence are ought to be bounded. We verify this stronger variant for split matroids, a class that was introduced only recently and generalizes paving matroids.  

Our main result is a proof of the conjecture for graphic matroids of wheels. Though wheels are structurally rather simple, the proof for this graph class is already non-trivial and requires a thorough understanding of feasible exchanges. As a byproduct, we show that the minimum number of steps required to transform $\bP_1$ into $\bP_2$ can be arbitrarily large compared to the lower bound $r-|R_1\cap R_2|$. 

Finally, we prove the conjecture for spikes, an important class of 3-connected matroids. Spikes are interesting because, as we show, one can define an arbitrarily large number of basis pairs without a strictly monotone exchange sequence between any two of them. This is in sharp contrast to the case of wheels, where for any three pairs of bases, there exist two with a strictly monotone exchange sequence between them.

\medskip

The rest of the paper is organized as follows. Basic notions and definitions are given in Section~\ref{sec:prelim}, together with some elementary observations on wheels. We verify Conjecture~\ref{conj:weighted} for strongly base orderable matroids and for split matroids in Section~\ref{sec:easy}. Graphic matroids of wheels are considered in Section~\ref{sec:wheels}, while spikes are discussed in Section~\ref{sec:spikes}. We close the paper with further observations in Section~\ref{sec:further}.

\section{Preliminaries}
\label{sec:prelim}

\paragraph{General notation.}

The set of nonnegative real numbers is denoted by $\mathbb{R}_+$. For subsets $X,Y\subseteq S$, their \textbf{symmetric difference} is defined as $X\triangle Y\coloneqq (X\setminus Y)\cup (Y\setminus X)$. When $Y$ consist of a single element $y$, then $X\setminus \{y\}$ and $X\cup\{y\}$ are abbreviated as $X-y$ and $X+y$, respectively. Given a weight function $w\colon S\to\mathbb{R}_+$ and a subset $X\subseteq S$, we use the notation $w(X)=\sum_{s\in X}w(s)$. 

\paragraph{Matroids.}

For basic definitions on matroids, we refer the reader to~\cite{oxley2011matroid}. If $M$ is a rank-$r$ matroid over a ground set $S$ of size $2r$ such that $S$ decomposes into two disjoint bases $R$ and $B$ of $M$, then such a decomposition is called a \textbf{coloring} of $M$\footnote{Throughout the paper, we will refer to the elements of $R$ and $B$ as `red' and `blue', respectively.}. A \textbf{feasible exchange} of elements $r\in R$ and $b\in B$ is denoted by $(r,b)$. We extend this notation to a \textbf{sequence of symmetric exchanges} as well by writing $(r_1,b_1),\dots,(r_k,b_k)$, meaning that $(R\setminus \{r_1,\dots,r_i\})\cup \{b_1,\dots,b_i\}$ and $(B\cup \{r_1,\dots,r_i\})\setminus \{b_1,\dots,b_i\}$ are bases for $i=1,\dots,k$. A matroid $M$ is \textbf{strongly base orderable} if  or any two bases $B_1, B_2$, there exists a bijection $\phi\colon B_1 \to B_2$ such that  $(B_1 \setminus X) \cup \phi(X)$ is also a basis for any $X \subseteq B_1$,  where we denote $\phi(X) \coloneqq \{ \phi(e) \mid e \in X\}$.

Let $S$ be a ground set of size at least $r$, $\cH=\{H_1,\dots, H_q\}$ be a (possibly empty) collection of subsets of $S$, and $r, r_1, \dots, r_q$ be nonnegative integers satisfying $|H_i \cap H_j| \le r_i + r_j -r$ for $1 \le i < j \le q$, and $|S\setminus H_i| + r_i \ge r$ for $i=1,\dots, q$. Then $\cI=\{\, X\subseteq S\mid |X|\leq r,\ |X\cap H_i|\leq r_i\ \text{for $1\leq i \leq q$} \,\}$ forms the family of independent sets of a rank-$r$ matroid $M$ that we call an \textbf{elementary split matroid}; see \cite{berczi2022hypergraph} for details. A set $F \subseteq S$ is called \textbf{$H_i$-tight} if $|F \cap H_i| = r_i$. A \textbf{split matroid} is the direct sum of a single elementary split matroid and some (maybe zero) uniform matroids.

For a graph $G=(V,E)$ on $n$ vertices, the \textbf{graphic matroid} $M=(E,\cI)$ of $G$ is defined on the edge set by considering a subset $F\subseteq E$ to be independent if it is a forest, that is, $\cI=\{F\subseteq E \mid F\ \text{does not contain a cycle}\}$. If the graph is connected, then the bases of the graphic matroid are exactly the spanning trees of $G$ and the rank of the matroid is $n-1$. 

Let $S=\{t,x_1,y_1,\dots,x_r,y_r\}$ be a ground set of size $2r+1$ for some $r\geq 3$, and let $\cC_1=\{\{t,x_i,y_i\} \mid 1\leq i\leq r\}$ and $\cC_2=\{\{x_i,y_i,x_j,y_j\} \mid 1\leq i<j\leq r\}$. Furthermore, let $\cC_3\subseteq \{Z\subseteq S \mid |Z|=r,\ |Z\cap\{x_i,y_i\}|=1\ \text{for $1\leq i\leq r$}\}$; note that $\cC_3$ might be empty. Finally, define $\cC_4=\{C\subseteq S \mid |C|=r+1,\ C'\not\subseteq C\ \text{for $C'\in\cC_1\cup\cC_2\cup\cC_3$}\}$. Then the family $\cC=\cC_1\cup\cC_2\cup\cC_3\cup\cC_4$ satisfies the circuit axioms, hence $M=(S,\cC)$ is a rank-$r$ matroid with circuit family $\cC$. Matroids arising this way are called \textbf{spikes}, where $t$ and the pairs $\{x_i,y_i\}$ are called the \textbf{tip} and the \textbf{legs} of the spike, respectively. It is not difficult to check that by restricting $M$ to any $2r$ of its elements (or in other words, deleting any of its elements) results in a matroid whose ground set decomposes into two disjoint bases.

\paragraph{Wheels.}

A graph $G=(V,E)$ is called a \textbf{wheel graph} (or \textbf{wheel} for short) if it is obtained by connecting a vertex, called the \textbf{center of the wheel}, to all the vertices of a cycle of length at least three, called the \textbf{outer cycle} of the wheel. In particular, wheels have at least four vertices. Edges connecting the center vertex with the vertices of the outer cycle are called \textbf{spokes}, while the edges of the outer cycle are called \textbf{rim edges}. Wheels are clearly planar, and so the order of the vertices on the outer cycle implies a natural cyclic ordering of the spokes and the rim edges as well.

It is not difficult to check that any wheel is the disjoint union of two spanning trees. Therefore, a coloring of the graphic matroid of a wheel is basically a partition of its edge set into two colors $R$ and $B$ such that both color classes form a spanning tree. A nice property of wheels is that we have a fairly good understanding of the structure of their colorings. Indeed, in order to decompose a wheel into two spanning trees, we first need to split the spokes into two nonempty sets. Then, if a rim edge goes between the endpoints of two spokes having the same color, then it automatically has to be colored with the other color to obtain a basis. Hence it only remains to decide the color of the rim edges going between the endpoints of spokes having different colors. However, once we fix the color of any of those edges, it determines the color of all the remaining ones. 

We call a maximal set of consecutive spokes of the same color an \textbf{interval}. By the \textbf{length} and \textbf{color of the interval} we mean the number and color of the spokes in it, respectively. Rim edges going between two intervals are called \textbf{boundary edges}. By the above, for $X\in\{R,B\}$, either every interval of color $X$ is followed by a boundary edge of color $X$ in a counterclockwise direction, or every interval of color $X$ is followed by a boundary edge of color $X$ in a clockwise direction. This property is referred to as the \textbf{orientation} of the coloring, and the orientation is called \textbf{positive} in the former and \textbf{negative} in the latter case; see Figure~\ref{fig:wheela} for an example. Orientations will play a crucial role in whether one can go from one coloring to another using a small number of exchanges or not.

\begin{figure}[t!]
\centering
\begin{subfigure}[t]{0.47\textwidth}
\centering
\includegraphics[width=.7\linewidth]{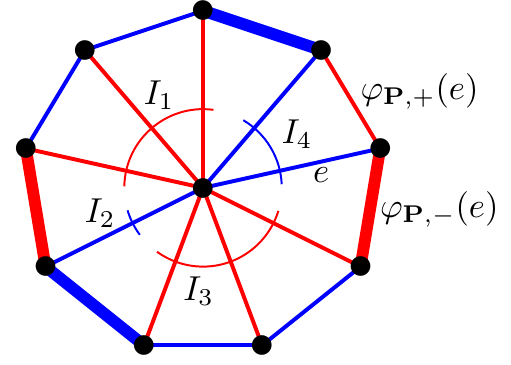}
\caption{The red and blue intervals are denoted by $I_1,I_3$ and $I_2,I_4$, respectively.}
\label{fig:wheela}
\end{subfigure}\hfill
\begin{subfigure}[t]{0.47\textwidth}
\centering
\includegraphics[width=.7\linewidth]{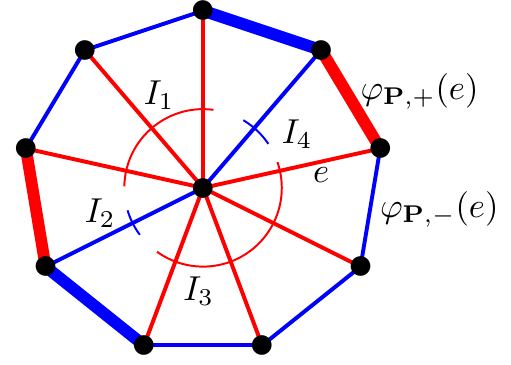}
\caption{The coloring obtained by symmetrically exchanging $e$ and its pair $\varphi_{\bP,-}(e)$.}
\label{fig:wheelb}
\end{subfigure}
\caption{Colorings containing four intervals. Thick rim edges correspond to boundary edges. Note that both colorings have positive orientation.}
\label{fig:wheel}
\end{figure}

Given a coloring $\bP=(R,B)$ of the wheel graph, each spoke $e$ is incident to two rim edges. The rim edge sharing a vertex with $e$ in the direction opposite of the orientation of the coloring can always be exchanged with $e$. We denote this rim edge by $\varphi_{\bP,-}(e)$, while the other rim edge incident to $e$ is denoted by $\varphi_{\bP,+}(e)$; see Figure~\ref{fig:wheelb} for an example. Both $\varphi_{\bP,-}$ and $\varphi_{\bP,+}$ provide bijections between spokes and rim edges. Note that these bijections are already determined by the orientation of $\bP$. If there are at least four intervals in the coloring, then the feasible exchanges are exactly the ones exchanging $e$ and $\varphi_{\bP,-}(e)$ for some spoke $e$. When there are only two intervals in the coloring, then there are other pairs that can be symmetrically exchanged.

\section{Strongly base orderable and split matroids}
\label{sec:easy}

As a warm-up, we first consider two basic cases: strongly base orderable and split matroids. For strongly base orderable matroids, the proof of Conjecture~\ref{conj:weighted} can be read out directly from the proof of White's conjecture in~\cite{lason2014toric} for strongly base orderable matroids. However, it might help the reader to get familiar with the notion of basis exchanges, and also sheds light to the difficulties caused by the presence of a weight function. For split matroids, the proof is more involved and relies heavily on that of Hamidoune's conjecture appeared in~\cite{berczi2022exchange}.  

\subsection{Strongly base orderable matroids}

\begin{thm}\label{thm:sbo}
Let $\bP_1 = (R_1, B_1)$ and $\bP_2 = (R_2, B_2)$ be compatible pairs of disjoint bases of a rank-$r$ strongly base orderable matroid $M$ over a ground set $S$, and let $w\colon S\to\mathbb{R}_+$ be a weight function. Then there exists a sequence of exchanges of total weight at most $w(R_1\cup B_1) = w(R_2\cup B_2)$ that transforms $\bP_1$ into $\bP_2$ and uses each element at most twice.
\end{thm}
\begin{proof}
Let $\phi_1\colon R_1 \to B_1$ and $\phi_2\colon R_2\to B_2$ be bijections such that $(R_i\setminus X)\cup\phi_i(X)$ is a basis for each $X\subseteq R_i$ and $i=1,2$. Consider the bipartite graph $G$ with vertex set $R_1\cup B_1=R_2\cup B_2$, and edges of the form $e\phi_1(e)$ for $e\in R_1$ and $f\phi_2(f)$ for $f\in R_2$. We denote the color classes of $G$ by $S$ and $T$. Note that 
\begin{equation*}
S=(S\cap R_1)\cup (S\cap B_1)=(R_1\setminus (R_1\setminus S))\cup\phi_1(R_1\setminus S)     
\end{equation*}
and
\begin{equation*}
T=(T\cap R_1)\cup (T\cap B_1)=(R_1\setminus (R_1\setminus T))\cup\phi_1(R_1\setminus T),
\end{equation*}
hence both $S$ and $T$ are bases of $M$. Let us define the basis pairs $\bP=(S,T)$ and $\bP'=(T,S)$. 

By exchanging the elements between $R_1\setminus S$ and $S\setminus R_1$ according to $\phi_1$, we get a sequence of weight $w(R_1\triangle S)$ that transforms $\bP_1$ into $\bP$. By exchanging the elements between $S\setminus R_2$ and $R_2\setminus S$ according to $\phi_2$, we get a sequence of weight $w(R_2\triangle S)$ that transforms $\bP$ into $\bP_2$. The concatenation of these two sequences transforms $\bP_1$ into $\bP_2$, has total weight $w(R_1\triangle S)+w(R_2\triangle S)$, and uses each element at most twice.

By exchanging the elements between $R_1\setminus T$ and $T\setminus R_1$ according to $\phi_1$, we get a sequence of weight $w(R_1\triangle T)$ that transforms $\bP_1$ into $\bP'$. By exchanging the elements between $T\setminus R_2$ and $R_2\setminus T$ according to $\phi_2$, we get a sequence of weight $w(R_2\triangle T)$ that transforms $\bP$ into $\bP_2$. The concatenation of these two sequences transforms $\bP_1$ into $\bP_2$, has total weight $w(R_1\triangle T)+w(R_2\triangle T)$, and uses each element at most twice.

Since
\begin{align*}
    & w(R_1\triangle S)+w(R_2\triangle S)+w(R_1\triangle T)+w(R_2\triangle T)\\
    {}={} & (w(R_1\triangle S)+w(R_1\triangle T))+(w(R_2\triangle S)+w(R_2\triangle T))\\
    {}={} & (w(R_1)+w(B_1))+(w(R_2)+w(B_2)),
\end{align*}
at least one of the above defined sequences has total weight at most $w(R_1\cup B_1)=w(R_2\cup B_2)$ and uses each element at most twice. This concludes the proof of the theorem.
\end{proof}

\subsection{Split matroids}

The introduction of split matroids was motivated by the study of matroid polytopes from a geometry point of view~\cite{joswig2017matroids}. Besides their immediate applications in tropical geometry, split matroids generalize paving matroids, a class that plays a fundamental role among matroids. Crapo and Rota~ \cite{crapo1970foundations} and Mayhew, Newman, Welsh and Whittle~\cite{mayhew2011asymptotic} conjectured that the asymptotic fraction of matroids on $n$ elements that are paving tends to $1$ as $n$ tends to infinity, hence the same is expected to hold for split matroids as well. 

In \cite{berczi2022exchange}, Bérczi and Schwarcz showed that Hamidoune's conjecture holds for split matroids. In fact, they proved that the exchange distance of compatible basis pairs $\bP_1=(A_1,A_2)$ and $\bP_2=(B_1,B_2)$ of a rank-$r$ split matroid is at most $\min\{r,r-|A_1\cap B_1|+1\}$, and that a shortest sequence transforming $\bP_1$ into $\bP_2$ can be found in polynomial time if the matroid in question is given by an independence oracle. By building on their proof, we show how to deduce a strengthening of their result.

\begin{thm} \label{thm:split}
Let $\bP_1 = (R_1,B_1)$ and $\bP_2 = (R_2,B_2)$ be compatible pairs of disjoint bases of a rank-$r$ split matroid $M$ over a ground set $S$, and let $w\colon  S\to \mathbb{R}_+$ be a weight function. Then there exists a sequence of exchanges of length at most  $\min\{r,r-|R_1\cap R_2|+1\}$ and total weight at most $w(R_1\cup B_1) = w(R_2\cup B_2)$ that transforms $\bP_1$ into $\bP_2$ and uses each element at most twice.
\end{thm}
\begin{proof}
Similarly to the proof in~\cite{berczi2022exchange}, we first show that it suffices to verify the theorem for elementary split matroids. If $M$ is not elementary, then all but at most one of its components are uniform matroids. Let $M_0=(S_0,\cB_0)$ denote, if exists, the elementary split component, and let $M_1=(S_1,\cB_1),\dots,M_t=(S_t,\cB_t)$ be the uniform connected components of $M$. Define $R_1^i\coloneqq R_1\cap S_i$, $B_1^i\coloneqq B_1\cap S_i$, $R_2^i\coloneqq R_2\cap S_i$, $B_2^i\coloneqq B_2\cap S_i$ for $i=0,\dots,t$. Clearly, $R_1^i$, $B_1^i$, $R_2^i$ and $B_2^i$ are bases of the uniform matroid $M_i$. Take a sequence of exchanges of length at most $\min\{r_0,r_0-|R_1^0\cap R_2^0|+1\}$ and total weight at most $w(R_1^0\cup B_1^0)$ that transforms $(R^0_1,B^0_1)$ into $(R^0_2,B^0_2)$ and uses each element in $S_0$ at most twice, where $r_0$ is the rank of $M_0$. For each uniform matroid $M_i$, take a strictly monotone sequence of exchanges that transforms $(R_1^i,B_1^i)$ into $(R_2^i,B_2^i)$; it is not difficult to see that such a sequence exists for uniform matroids. The concatenations of these sequences result in an exchange sequence of length at most $\min\{r_0,r_0-|R_1^0\cap R_2^0|+1\}+\sum_{i=1}^t [r_i-|R_1^i\cap R_2^i|]\leq \min\{r,r-|R_1\cap R_2|+1\}$ and total weight at most $w(R_1\cup B_1)=w(R_2\cup B_2)$ that transforms $(R_1,B_1)$ into $(R_2,B_2)$ and uses each element at most twice. 

By the above, we may assume that $M$ is an elementary split matroid.  Let $\cH=\{H_1,\dots,H_q\}$ be the hypergraph that defines the matroid, and let $r_1,\dots,r_q$ denote the upper bounds on the size of the intersection of an independent set with the hyperedges. The length and the weight of any strictly monotone sequence of exchanges that transforms $\bP_1$ into $\bP_2$ has length $r-|R_1\cap R_2|$ and weight $w(R_1\triangle R_2)\leq w(R_1\cup B_1)=w(R_2\cup B_2)$, respectively, and uses each element at most once. Hence assume that this is not the case, that is, there exists no strictly monotone exchange sequence between $\bP_1$ and $\bP_2$. In such a case, \cite{berczi2022exchange} showed that $\min\{r, r-|R_1 \cap R_2|+1\} = r-|R_1 \cap R_2|+1$ and after performing a longest strictly monotone exchange sequence from $\bP_1$ with respect to $\bP_2$, one is left with a basis pair $\bP'_1=(R'_1,B'_1)$ for which there exists pairwise distinct hyperedges $H_1$, $H_2$, $H_3$ and $H_4$ in $\cH$ with the following properties:\footnote{Here \ref{it:tight} follows from Claims 8 and 11, \ref{it:sym} follows from Claims 9 and 10, and \ref{it:rest} follows from Claims 9 and 11 of \cite{berczi2022exchange}.} 
\begin{enumerate}[label=\alph*)]
\item \label{it:tight} $R_1$, $R'_1$ and $R_2$ are $H_1$- and $H_3$-tight, while $B_1$, $B'_1$ and $B_2$ are $H_2$- and $H_4$-tight,
\item \label{it:sym} $(R_1 \cap B_2) \cup (B_1 \cap R_2) \subseteq (H_1 \triangle H_3)\cap (H_2 \triangle H_4)$,
\item \label{it:rest} $|R'_1 \cap B_2 \cap H_1 \cap H_2| = |R'_1 \cap B_2 \cap H_3 \cap H_4| = |B'_1 \cap R_2 \cap H_1 \cap H_4| = |B'_1 \cap R_2 \cap H_2 \cap H_3| > 0$.
\end{enumerate}
Moreover, for each $z \in ((R_1 \cap R_2) \cup (B_1 \cap B_2)) \setminus ((H_1 \triangle H_3)\cap (H_2 \triangle H_4))$, \cite{berczi2022exchange} showed the existence of an exchange sequence that transforms $\bP'_1$ into $\bP_2$, uses $z$ twice, uses the elements of $(R'_1 \cap B_2) \cup (B'_1 \cap R_2)$ once, and does not use any other elements. Consider the concatenation of the strictly monotone exchange sequence from $\bP_1$ to $\bP'_1$ and the above exchange sequence from $\bP'_1$ to $\bP_2$. The sequence thus obtained uses each element at most twice, has length $r-|R_1 \cap R_2| + 1$ and total weight \[2 w(z) + w((R_1 \cap B_2)\cup (R_2 \cap B_1)) = w(R_1 \cup B_1) + 2w(z) - w((R_1 \cap R_2) \cup (B_1 \cap B_2)).\] Therefore, it suffices to show that there exists an element $z \in ((R_1 \cap R_2) \cup (B_1 \cap B_2)) \setminus ((H_1 \triangle H_3) \cap (H_2 \cap H_4))$ with weight $w(z) \le \frac12 w((R_1 \cap R_2) \cup (B_1 \cap B_2))$. We prove that
\begin{equation*}
  \tag*{($\star$)}\label{star} |((R_1 \cap R_2) \cup (B_1 \cap B_2)) \setminus ((H_1 \triangle H_3)\cap(H_2 \triangle H_4))| \ge 2
\end{equation*}
holds from which the statement follows. Using properties \ref{it:sym} and \ref{it:rest}, we get
\begin{align*}
0 
{}&{}< 
|B'_1 \cap R_2 \cap H_2 \cap H_3|\\
{}&{}\le 
|B_1 \cap R_2 \cap H_2 \cap H_3|\\
{}&{}=
|(B_1 \cap R_2) \setminus (H_1 \cup H_4)|\\
{}&{}\le
|B_1 \setminus (H_1 \cup H_4)|.    
\end{align*}
By \cite[Lemma~6]{berczi2022exchange}, if a set $F$ of size $r$ is both $H$-tight and $H'$-tight for some hyperedges $H, H'\in \mathcal{H}$ then $H \cap H' \subseteq F \subseteq H \cup H'$.  As $B_1$ is $H_4$-tight and $B_1 \not \subseteq H_1 \cup H_4$, it follows that $B_1$ is not $H_1$-tight, that is, $|B_1 \cap H_1| \le r_1-1$. Using that $R_2$ is $H_1$-tight, we get
\begin{align*}
|B_1 \cap B_2 \cap H_1|
{}&{}= |B_1 \cap H_1| - |B_1 \cap R_2 \cap H_1| \\
{}&{}\le (r_1 -1) - |B_1 \cap R_2 \cap H_1| \\
{}&{}= |R_2 \cap H_1| - |B_1 \cap R_2 \cap H_1| - 1\\
{}&{}= |R_1 \cap R_2 \cap H_1| -1.
\end{align*}
An analogous reasoning shows that the inequality \[|B_1 \cap B_2 \cap H_3| \le |R_1 \cap R_2 \cap H_3| -1\] holds as well.
As $R_1$ is both $H_1$- and $H_3$-tight, $H_1 \cap H_3 \subseteq R_1 \subseteq H_1 \cup H_3$ holds by \cite[Lemma~6]{berczi2022exchange}. Then 
\begin{align*}
    |B_1 \cap B_2| - |(B_1 \cap B_2) \setminus (H_1 \cup H_3)| {}&{} = |B_1 \cap B_2 \cap H_1| + |B_1 \cap B_2 \cap H_3| \\
    {}&{} \le |R_1 \cap R_2 \cap H_1| + |R_1 \cap R_2 \cap H_3| -2 \\
    {}&{} = |R_1 \cap R_2| + |R_1 \cap R_2 \cap H_1 \cap H_3| -2.
\end{align*}
Using $|B_1 \cap B_2| = |R_1 \cap R_2|$, we get
\begin{align*}
2
{}&{}\le 
|R_1 \cap R_2 \cap H_1 \cap H_3| + |(B_1 \cap B_2) \setminus (H_1 \cup H_3)|\\
{}&{}\le 
|((R_1 \cap R_2) \cup (B_1 \cap B_2)) \setminus (H_1 \triangle H_3)|,    
\end{align*}
implying~\ref{star} and thus concluding the proof of the theorem.
\end{proof}

\section{Wheels}
\label{sec:wheels}

Our first main result is a proof of Conjecture~\ref{conj:weighted} for the graphic matroid of wheels.  In fact, we prove a much stronger statement: we verify that for any pair $\bP_1=(R_1,B_1)$ and $\bP_2=(R_2,B_2)$ of colorings of a wheel $G=(V,E)$, there exists a sequence of exchanges of length at most $r$ and total weight at most $w(E)$ that transforms $\bP_1$ into $\bP_2$ and uses each edge at most twice.

Throughout the section, we assume that $\bP_1$ has positive orientation. For ease of notation, we introduce $\varphi_\oplus\coloneqq \varphi_{\bP_1,+}$ and $\varphi_\ominus\coloneqq\varphi_{\bP_1,-}$. Recall that both $\varphi_\ominus$ and $\varphi_\oplus$ provide a bijection between spokes and rim edges.

First we settle the case when the two colorings have the same orientation.

\begin{lem}\label{lem:same}
Let $\bP_1=(R_1,B_1)$ and $\bP_2=(R_2,B_2)$ be colorings of a wheel $G=(V,E)$ with the same orientation. Then there exists a strictly monotone sequence of exchanges that transforms $\bP_1$ into $\bP_2$.
\end{lem}
\begin{proof}
Exchange each spoke $e$ that has different color in $\bP_1$ than in $\bP_2$ with its pair $\varphi_\ominus(e)$ in an arbitrary order, only paying attention to always have at least one spoke in both color classes. Once the spokes have the right colors, that is, they are colored as in $\bP_2$, the rim edges are also colored as required. Indeed, the orientation was not changed during the procedure and the coloring of the spokes together with the orientation of the coloring uniquely determines the colors of the rim edges. 
\end{proof}

\begin{rem}
Exchanging a spoke $e$ with its pair $\varphi_\ominus(e)$ does not reverse the orientation of the coloring. Therefore, if the initial coloring $\bP_1=(R_1,B_1)$ and the target coloring $\bP_2=(R_2,B_2)$ have different orientations, then any sequence of feasible exchanges must go through a coloring of the graph with only two intervals. This results in reversing the orientation of a coloring to be costly, as one needs to make exchange steps that, at least seemingly, bring the colorings further away from each other in the sense that they increase the symmetric difference of $R_1$ and $R_2$. The class of wheels is particularly interesting as they admit colorings $\bP_1=(R_1,B_1)$ and $\bP_2=(R_2,B_2)$ for which the minimum number of steps required to transform $\bP_1$ into $\bP_2$ is arbitrarily large compared to the trivial lower bound $(n-1)-|R_1\cap R_2|$; see Figure~\ref{fig:large} for an example.
\end{rem}

\begin{figure}[t!]
\centering
\begin{subfigure}[t]{0.47\textwidth}
\centering
\includegraphics[width=.6\linewidth]{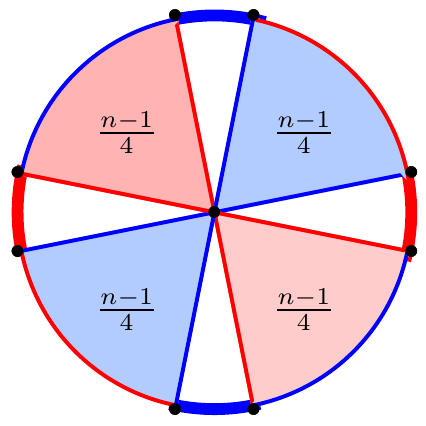}
\caption{A coloring $\bP_1=(R_1,B_1)$ with positive orientation. Thick edges denote the boundary edges.}
\label{fig:gapa}
\end{subfigure}\hfill
\begin{subfigure}[t]{0.47\textwidth}
\centering
\includegraphics[width=.6\linewidth]{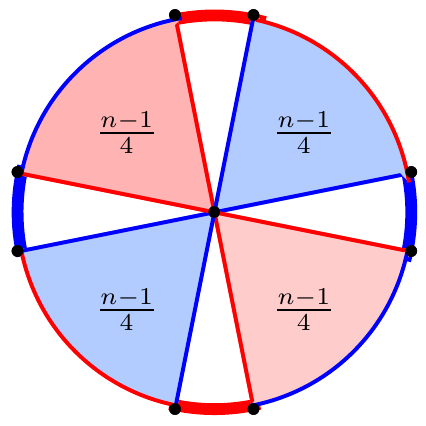}
\caption{A coloring $\bP_2=(R_2,B_2)$ with negative orientation. Thick edges denote the boundary edges.}
\label{fig:gapb}
\end{subfigure}
\caption{A pair of $\bP_1=(R_1,B_1)$ and $\bP_2=(R_2,B_2)$ with $(n-1)-|R_1\cap R_2|=2$. The colorings have opposite orientations, hence any sequence of exchanges that transforms $\bP_1$ into $\bP_2$ goes through a coloring having two intervals, implying that their exchange distance is at least $\frac{n-1}{4}$.}
\label{fig:large}
\end{figure}

Next we consider colorings with different orientations and a bounded number of intervals in one of the color classes.

\begin{lem}\label{lem:24}
Let $\bP_1=(R_1,B_1)$ and $\bP_2=(R_2,B_2)$ be colorings of a wheel $G=(V,E)$ with different orientations where $\bP_1$ has at most four intervals, and let $w\colon E\to\mathbb{R}_+$ be a weight function. Then there exists a sequence of exchanges of length at most $n-1$ and total weight at most $w(E)$ that transforms $\bP_1$ into $\bP_2$ and uses each edge at most twice.
\end{lem}
\begin{proof} 
We will distinguish two main cases based on the number of intervals in $\bP_1$. Recall that $\bP_1$ is assumed to have positive orientation throughout.

\medskip
\noindent \textbf{Case 1.} $\bP_1$ has two intervals.

\smallskip
\noindent \textbf{Case 1.1.} $\bP_1$ has an interval of length one.

We may assume that there exists a red interval of length one, and let $c$ denote the unique spoke in it. We denote by $a$ the spoke following $c$ in negative direction, and further define $b \coloneqq  \varphi_\oplus(a) = \varphi_\ominus(c)$ and $d\coloneqq \varphi_\oplus(c)$. Note that $a$ and $b$ are blue, while $d$ is a red edge; see Figure~\ref{fig:case11}. As $\bP_2$ has negative orientation, $a$ and $\varphi_\oplus(a) = b$ have different colors in $\bP_2$, and the same holds for $c$ and $\varphi_\oplus(c)=d$. Hence the set of edges among $a$, $b$, $c$ and $d$ that have different colors in $\bP_1$ and $\bP_2$ is either $\{a,c\}, \{a,d\}, \{b,d\}$ or $\{b,c\}$. In the first three cases, changing the color of the two edges is a feasible exchange which reverses the orientation. Once the orientation of the coloring is reversed, there exists a strictly monotone exchange sequence to $\bP_2$ by Lemma~\ref{lem:same}, altogether resulting in a strictly monotone exchange sequence from $\bP_1$ to $\bP_2$.

The only remaining case is when the set of edges among $a$, $b$, $c$ and $d$ that need to change color is $\{b,c\}$. In this case, the difficulty comes from the fact that these edges do not define a feasible exchange between the two color classes. In order to overcome this, extra steps are needed to reverse the orientation. Let $s$ be an arbitrary red spoke of $\bP_2$. Note that $s \not\in \{a,c\}$ as we are in the case when $a$ and $c$ are blue in $\bP_2$. As $\bP_1$ has a unique red spoke, namely $c$, we get that $s$ is blue in $\bP_1$, and $\varphi_\oplus(s)$ is red in $\bP_1$ and blue in $\bP_2$ since $\bP_2$ has negative orientation. Consider the two exchange sequences of length three $(b,d)$, $(s,\varphi_\oplus(s))$, $(c,d)$ and $(a,c)$, $(s,\varphi_\oplus(s))$, $(a,b)$. Both of these sequences reverse the orientation of the coloring and fix the colors of the edges $a$, $b$, $c$ and $d$. Therefore, after applying any of them, there exists a strictly monotone exchange sequence to $\bP_2$ by Lemma~\ref{lem:same} that uses all the remaining edges in $E-\{a,b,c,d\}$ at most once. Thus in overall, we get an exchange sequence that uses each edge in $E-\{a,d\}$ at most once, does not use one of $a$ and $d$ and uses the other twice. Hence the length of the sequence is at most half of the number of edges, that is, $n-1$. If $w(a)\geq w(d)$, then starting the sequence with $(b,d)$, $(s,\varphi_\oplus(s))$, $(c,d)$, while if $w(a)<w(d)$, then starting the sequence with $(a,c)$, $(s,\varphi_\oplus(s))$, $(a,b)$ ensures that total weight of the exchange sequence is at most $w(E)$, concluding the proof of the case.

\smallskip
\noindent \textbf{Case 1.2.} Both intervals of $\bP_1$ have length at least two.

Let $c$ denote the last spoke of the red interval in positive direction and let $d \coloneqq \varphi_\oplus(c)$. Furthermore, let $a$ be the last spoke of the blue interval in positive direction and let $b \coloneqq \varphi_\oplus(a)$, see Figure~\ref{fig:case12}. Similarly to Case~1.1, the set of edges among $a$, $b$, $c$ and $d$ that have different colors in $\bP_1$ and $\bP_2$ is either $\{a,c\}$, $\{a,d\}$, $\{b,d\}$ or $\{b,c\}$. However, now fixing the orientation is even simpler than before as any of the exchanges $(a,c)$, $(a,d)$, $(b,c)$ and $(b,d)$ is feasible. After reversing the orientation using one of these exchanges, there exists a strictly monotone exchange sequence to $\bP_2$ by Lemma~\ref{lem:same}, altogether resulting in a strictly monotone exchange sequence from $\bP_1$ to $\bP_2$.

\medskip

\begin{figure}[t!]
\centering
\begin{subfigure}[t]{0.23\textwidth}
\centering
\includegraphics[width=.95\linewidth]{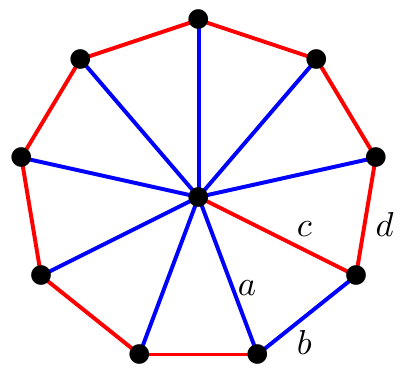}
\caption{$\bP_1$ contains two intervals, one of them having length one.}
\label{fig:case11}
\end{subfigure}\hfill
\begin{subfigure}[t]{0.23\textwidth}
\centering
\includegraphics[width=.95\linewidth]{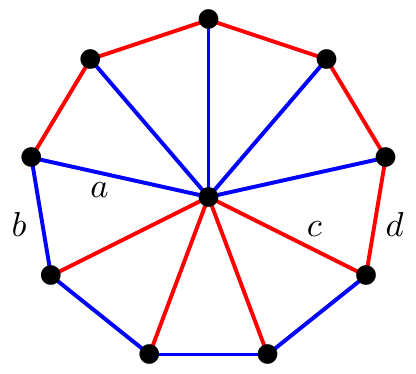}
\caption{$\bP_1$ contains two intervals, both having length at least two.}
\label{fig:case12}
\end{subfigure}\hfill
\begin{subfigure}[t]{0.23\textwidth}
\centering
\includegraphics[width=.95\linewidth]{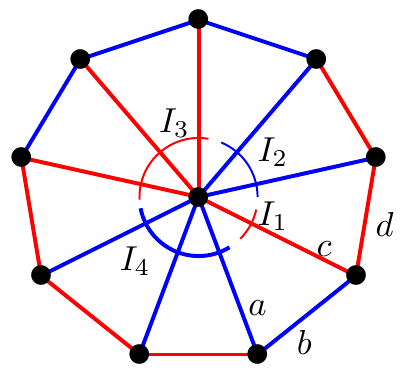}
\caption{$\bP_1$ contains four intervals, one of them having length one.}
\label{fig:case22a}
\end{subfigure}\hfill
\begin{subfigure}[t]{0.23\textwidth}
\centering
\includegraphics[width=.95\linewidth]{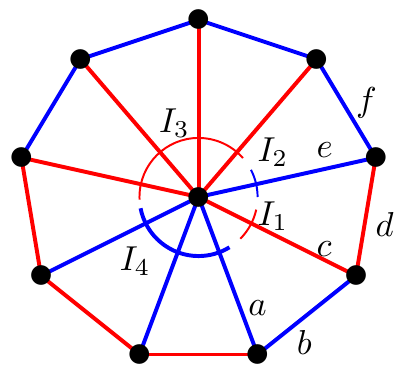}
\caption{$\bP_1$ contains both a red and a blue interval of length one.}
\label{fig:case22b}
\end{subfigure}
\caption{Illustration of the cases in the proof of Lemma~\ref{lem:24}.}
\label{fig:case1}
\end{figure}

We denote the number of spokes in $R_1$, $R_2$, $B_1$ and $B_2$ by $r_1$, $r_2$, $b_1$ and $b_2$, respectively. Let $2q$ denote the number of intervals in $\bP_1$, and let $I_1,\dots,I_{2q}$ denote the intervals in a positive direction, where intervals with odd indices have color red and intervals with even indices have color blue. Furthermore, for $1\leq i\leq 2q$, we define
\begin{eqnarray*}
&x_i\coloneqq\sum[w(e)\mid e\in I_i\cup\varphi_\ominus(I_i),\ e\ \text{has the same color in $\bP_1$ and $\bP_2$}],&\\
&y_i\coloneqq\sum[w(e)\mid e\in I_i\cup\varphi_\ominus(I_i),\ e\ \text{has different colors in $\bP_1$ and $\bP_2$}].&
\end{eqnarray*}
By the above definitions, we have $w(I_i\cup\varphi_\ominus(I_i))=x_i+y_i$ for $1\leq i\leq 2q$, and $\sum_{i=1}^{2q}(x_i+y_i)=w(E)$.

\medskip
\noindent \textbf{Case 2.} $\bP_1$ has four intervals.

Since $r_1+r_2+b_1+b_2=2(n-1)$, we have $\min \{r_1+r_2, b_1+b_2\} \le n-1$. We may assume that $r_1+r_2 \leq n-1$. We distinguish two cases based on the structure of the red intervals in $\bP_1$.

\smallskip
\noindent \textbf{Case 2.1.} $\bP_1$ has no red interval of length one.

Take one of the intervals $I_j \in \{I_1, I_3\}$ and exchange each spoke $p$ in $I_j$ with its pair $\varphi_\ominus(p)$. 
The coloring thus obtained has only two intervals from which the red has length more than one, hence there exists a strictly monotone exchange sequence to $\bP_2$ as in Case 1.2. Note that each edge is used at most twice during this process, and the edges used twice are exactly the edges of $I_j$ having the same color in $\bP_1$ and $\bP_2$. The length of the exchange sequence thus obtained is
\begin{align*}
|I_j| + |(I_j \cup I_2 \cup I_4) \cap R_2| + |I_{4-j} \cap B_2| 
{}&{} \le  
|I_j| + |(I_1 \cup I_2 \cup I_3 \cup I_4) \cap R_2| + |I_{4-j}| \\
{}&{} = 
r_1 + r_2 \\
{}&{}
\le 
n-1.
\end{align*}
In order to bound the total weight of the exchanges, note that we used edges twice only from $I_j\cup \varphi_\ominus(I_j)$, and all the other edges were used only if their color had to change. Thus the total weight of the exchange sequence is $2x_j+y_1+y_2+y_3+y_4$. As $2 \min \{x_1,x_3\} \le x_1 + \dots + x_4$, we get $2\min\{x_1,x_3\}+y_1+y_2+y_3+y_4\leq w(E)$. Hence choosing $j\in\{1,3\}$ such that $x_j=\min\{x_1,x_3\}$ leads to an exchange sequence with the required properties. 

\smallskip
\noindent \textbf{Case 2.2.} $\bP_1$ has a red interval of length one.

Our approach is similar to that of Case 2.1 with slight modifications. By applying the same algorithm, the analysis goes through if the the red interval whose edges are not used before reversing the orientation has length at least two. Indeed, the bounds on the length and weight of the exchange sequence remain valid since we did not use the length of the intervals in the proof. However, a problem occurs if at the orientation reversal step we have an interval of length one, that is, when the algorithm starts with a red interval, say $I_3$ due to $x_3\leq x_1$, and the other red interval $I_1$ has length one. 

Let $\bP'_1$ denote the coloring obtained by the exchanges so far, that is, by exchanging each spoke $p$ in $I_3$ with its pair $\varphi_\ominus(p)$. At this point, our goal is to reverse the orientation and then find a strictly monotone sequence to $\bP_2$. Note that this is exactly the same scenario we had in Case 1.1, hence we use an analogous notation that we had there. That is, let $c$ denote the unique red spoke of $\bP'_1$ and $a$ be the spoke following $c$ in a negative direction. Furthermore, define $b\coloneqq\varphi_\oplus(a)=\varphi_\ominus(c)$ and $d \coloneqq \varphi_\oplus(c)$, see Figure~\ref{fig:case22a}. Similarly to Case~1.1, the set of edges among $a$, $b$, $c$ and $d$ that have different colors in $\bP'_1$ and $\bP_2$ is either $\{a,c\}$, $\{a,d\}$, $\{b,d\}$ or $\{b,c\}$. In the first three cases, changing the color of the two edges is a feasible exchange which reverses the orientation and fixes the color of all four edges $a$, $b$, $c$ and $d$. Once the coloring is reversed, there exists a strictly monotone exchange sequence to $\bP_2$ by Lemma~\ref{lem:same}. The length and weight of the exchange sequence can be bounded analogously to Case~2.1.

The only remaining case is when the set of edges among $a$, $b$, $c$ and $d$ that need to change color is $\{b,c\}$. Let $s$ be an arbitrary red spoke of $\bP_2$. Note that $s \not\in \{a,c\}$ as we are in the case when $a$ and $c$ are blue in $\bP_2$. As $\bP'_1$ has a unique red spoke, namely $c$, we get that $s$ is blue in $\bP'_1$, and $\varphi_\oplus(s)$ is red in $\bP'_1$ and blue in $\bP_2$ since $\bP_2$ has negative orientation. 
Consider the two exchange sequences of length three $(b,d)$, $(s,\varphi_\oplus(s))$, $(c,d)$ and $(a,c)$, $(s,\varphi_\oplus(s))$, $(a,b)$. Both of these sequences change the orientation of the coloring and fix the colors of the edges $a$, $b$, $c$ and $d$. Therefore, after applying any of them, there exists a strictly monotone exchange sequence to $\bP_2$ by Lemma~\ref{lem:same}. If $w(a)\geq w(d)$, then use the sequence $(b,d)$, $(s,\varphi_\oplus(s))$, $(c,d)$ for reversing the orientation, and use $(a,c)$, $(s,\varphi_\oplus(s))$, $(a,b)$ otherwise. In what follows, we bound the length and weight of the sequence obtained.

Suppose first that $r_1+r_2<n-1$. The length of the exchange sequence is \[|I_3| + 3 + |(I_2 \cup I_3 \cup I_4 -s) \cap R_2| = (r_1-1) + 3 + (r_2-1) = r_1 + r_2 + 1 \le n-1.\] We used edges twice only from $I_3 \cup \varphi_\ominus(I_3)\cup\{d\}$ if $w(a) \ge w(d)$, and only from $I_3 \cup \varphi_\ominus(I_3)\cup \{a\}$ otherwise. Thus the total weight of the exchange sequence is at most \[2x_3 + 2 \min\{w(a), w(d)\} + y_1 + y_2 + y_3 + y_4 \le 2x_3 + x_2 + x_4 + y_1 + y_2 + y_3 + y_4 \le w(E),\] where we used that $2\min \{w(a), w(d)\} \le 2 \min \{x_4, x_2 \} \le x_2 + x_4$ and $x_3 \le x_1$.

We are left with the case when $r_1+r_2=n-1$, there exists a red interval consisting of a single spoke $c$, and the edges $c$ and $\varphi_\ominus(c)$ have different colors in $\bP'_1$ and $\bP_2$. Then, by $r_1+r_2+b_1+b_2=2n-2$, we get that $b_1+b_2=n-1$ also holds. Up to this point the two colors played different roles in the proof of Case 2 as we used the fact that $r_1+r_2\leq n-1$. However, now the same inequality holds for the number of blue spokes as well, hence we can switch the roles of the two colors. In particular, if none of the blue intervals have length one, then we are done. Thus suppose that the red interval $I_1$ consisting of the single spoke $c$ is followed in positive direction by a blue interval $I_2$ consisting of the single spoke $e$; the case when a blue interval of length one is followed in positive direction by a red interval of length one can be proved analogously. We denote by $a$ the spoke before $c$ in negative direction, and further define $b \coloneqq  \varphi_\oplus(a) = \varphi_\ominus(c)$, $d\coloneqq \varphi_\oplus(c)=\varphi_\ominus(e)$ and $f\coloneqq\varphi_\oplus(e)$, see Figure~\ref{fig:case22b}. Recall that we are in the case when the set of edges among $a$, $b$, $c$ and $d$ that need to change color is $\{b,c\}$. However, this means that among the edges $c$, $d$, $e$ and $f$, the set of edges that need to change color cannot be $\{d,e\}$. That is, if we start by exchanging each spoke $p$ in $I_4$ with its pair $\varphi_\ominus(p)$, then the resulting coloring $\bP''_1$ can be transformed into $\bP_2$ by a strictly monotone sequence of exchanges, and the analysis of the second paragraph of Case 2.2 applies.
\end{proof}

Our last technical lemma shows that when one of the colorings has at least six intervals, then there exists a sequence of exchanges that has low weight with respect to two arbitrary weight functions simultaneously.

\begin{lem}\label{lem:more}
Let $\bP_1=(R_1,B_1)$ and $\bP_2=(R_2,B_2)$ be colorings of a wheel $G=(V,E)$ with different orientations such that $\bP_1$ has at least six intervals, and let $w_1,w_2\colon E\to\mathbb{R}_+$ be weight functions. Then there exists a sequence of exchanges of total $w_i$-weight at most $w_i(E)$ for $i=1,2$ that transforms $\bP_1$ into $\bP_2$ and uses each edge at most twice.
\end{lem}
\begin{proof}
We distinguish two cases based on the remainder of the number of intervals modulo four.

\medskip
\noindent \textbf{Case 1.} $q=2k+1$ for some integer $k\geq 1$.

For an index $1\leq j\leq 4k+2$, exchange each spoke $e\in \bigcup_{i=1}^k I_{j+2i-1}$ with its pair $\varphi_\ominus(e)$, and do the same for each spoke $e\in \bigcup_{i=k+1}^{2k} I_{j+2i}$. After these exchanges, the resulting coloring $\bP'_1$ has two intervals: $I_j \cup I_{j+1} \cup \dots \cup I_{j+2k}$ has the same color in $\bP'_1$ as $I_j$ in $\bP_1$, and $I_{j+2k+1} \cup I_{j+2k+2} \cup \dots \cup I_{j+4k+1}$ has the other color. Note that none of these two intervals has length one as $k \ge 1$. Therefore, there exists a strictly monotone exchange sequence from $\bP'_1$ to $\bP_2$ by Case~1.2 of Lemma~\ref{lem:24}. Let $w\in\{w_1,w_2\}$, and let us define $I_i$, $x_i$ and $y_i$ for $1\leq i\leq 2q$ as in the proof of Lemma~\ref{lem:24}, where the $x_i$ and $y_i$ values are computed with respect to $w$. Our goal is to bound the $w$-weight of the above defined sequence of exchanges.

Exchanging each spoke $e$ in $\bigcup_{i=1}^k I_{j+2i-1}\cup  \bigcup_{i=k+1}^{2k} I_{j+2i}$ with its pair $\varphi_\ominus(e)$ has weight
\begin{equation*}
\sum_{i=1}^{k} (x_{j+2i-1} + y_{j+2i-1}) + \sum_{i=k+1}^{2k} (x_{j+2i}+y_{j+2i}). 
\end{equation*}
Then the strictly monotone sequence to $\bP_2$ has weight
\begin{equation*}
\sum_{i=0}^{k} y_{j+2i} + \sum_{i=1}^{k} x_{j+2i-1}  + \sum_{i = k}^{2k} y_{j+2i+1} + \sum_{i=k+1}^{2k} x_{j+2i}.
\end{equation*}
The total weight is then
\begin{equation*}
    2 \cdot \left (\sum_{i=1}^{k} x_{j+2i-1}+ \sum_{i=k+1}^{2k} x_{j+2i}\right)+\sum_{i = 1}^{4k+2}y_i.
\end{equation*}
Therefore the total $w$-weight of the exchange sequence is at most $w(E)=\sum_{i=1}^{4k+2}(x_i + y_i)$ if and only if 
\begin{equation}\label{eq:j}
    \sum_{i=1}^{k} x_{j+2i-1}+ \sum_{i=k+1}^{2k} x_{j+2i}\le \sum_{i=0}^{k} x_{j+2i}+ \sum_{i=k}^{2k} x_{j+2i+1}.\tag{$\ineqa_w(j)$}
\end{equation}

Consider inequalities $\ineqa_w(j)$ and $\ineqa_w(j+1)$. The sum of these two inequalities gives
\begin{equation*}
    \left(\sum_{i=1}^{4k+2} x_i\right)-(x_j+x_{j+2k+1})\le \left(\sum_{i=1}^{4k+2} x_i\right)+(x_j+x_{j+2k+1}).
\end{equation*}
As this inequality clearly holds, at least one of $\ineqa_w(j)$ and $\ineqa_w(j+1)$ must hold as well. Furthermore, $\ineqa_w(j)$ is identical to $\ineqa_w(j+2k+1)$. These together imply that $\ineqa_w(j)$ holds for at least $k+1$ choices of $j$ from $\{1,\dots,2k+1\}$ for $w\in\{w_1,w_2\}$. Therefore, there exists an index $j$ for which both $\ineqa_{w_1}(j)$ and $\ineqa_{w_2}(j)$ are satisfied. As each edge is used at most twice, the statement follows.

\smallskip
\noindent \textbf{Case 2.} $q=2k$ for some integer $k\geq 2$.

Our approach is similar to that of Case 1. For an index $1\leq j\leq 4k$, exchange each spoke $e\in \bigcup_{i=1}^{k-1} I_{j+2i-1}$ with its pair $\varphi_\ominus(e)$, and do the same for each spoke $e\in \bigcup_{i=k}^{2k-1} I_{j+2i}$. After these exchanges, the resulting coloring $\bP'_1$ has two intervals: $I_j \cup I_{j+1} \cup \dots \cup I_{j+2k-1}$ has the same color in $\bP'_1$ as $I_j$ in $\bP_1$, and $I_{j+2k} \cup I_{j+2k+1} \cup \dots \cup I_{j+4k}$ has the other color. Note that none of these two intervals has length one as $k \ge 2$. Therefore, there exists a strictly monotone exchange sequence from $\bP'_1$ to $\bP_2$ by Case~1.2 of Lemma~\ref{lem:24}. Let $w\in\{w_1,w_2\}$, and let us define $I_i$, $x_i$ and $y_i$ for $1\leq i\leq 2q$ as in the proof of Lemma~\ref{lem:24}, where the $x_i$ and $y_i$ values are computed with respect to $w$. Our goal is to bound the $w$-weight of the above defined sequence of exchanges.

Exchanging each spoke $e$ in $\bigcup_{i=1}^{k-1} I_{j+2i-1}\cup  \bigcup_{i=k}^{2k-1} I_{j+2i}$ with its pair $\varphi_\ominus(e)$ has weight
\begin{equation*}
\sum_{i=1}^{k-1} (x_{j+2i-1} + y_{j+2i-1}) + \sum_{i=k}^{2k-1} (x_{j+2i}+y_{j+2i}). 
\end{equation*}
Then the strictly monotone sequence to $\bP_2$ has weight
\begin{equation*}
\sum_{i=0}^{k-1} y_{j+2i} + \sum_{i=1}^{k-1} x_{j+2i-1}  + \sum_{i = k}^{2k} y_{j+2i-1} + \sum_{i=k}^{2k-1} x_{j+2i}.
\end{equation*}
The total weight is then
\begin{equation*}
    2 \cdot \left (\sum_{i=1}^{k-1} x_{j+2i-1}+ \sum_{i=k}^{2k-1} x_{j+2i}\right)+\sum_{i = 1}^{4k}y_i.
\end{equation*}
Therefore the total $w$-weight of the exchange sequence is at most $w(E)=\sum_{i=1}^{4k}(x_i + y_i)$ if and only if 
\begin{equation}\label{eq:jj}
    \sum_{i=1}^{k-1} x_{j+2i-1}+ \sum_{i=k}^{2k-1} x_{j+2i}\le \sum_{i=0}^{k-1} x_{j+2i}+ \sum_{i=k-1}^{2k-1} x_{j+2i+1}.\tag{$\ineqb_w(j)$}
\end{equation}

Consider inequalities $\ineqb_w(j)$ and $\ineqb_w(j+1)$. The sum of these two inequalities gives
\begin{equation*}
    \left(\sum_{i=1}^{4k} x_i\right)-(x_{j}+x_{j+2k-1})\le \left(\sum_{i=1}^{4k} x_i\right)+(x_{j}+x_{j+2k-1}).
\end{equation*}
As this inequality clearly holds, at least one of $\ineqb_w(j)$ and $\ineqb_w(j+1)$ must hold as well. This implies that $\ineqb_w(j)$ holds for at least $2k$ choices of $j$ from $\{1,\dots,4k\}$. Note that if the number of such choices is exactly $2k$, then $\ineqb_w(j)$ holds either for all odd or for all even indices.

Now consider inequalities $\ineqb_w(j)$ and $\ineqb_w(j+2k)$. The sum of these two inequalities gives
\begin{equation*}
    \left(\sum_{i=1}^{4k} x_i\right)-(x_{j+2k-1}+x_{j+4k-1})\le \left(\sum_{i=1}^{4k} x_i\right)+(x_{j+2k-1}+x_{j+4k-1}).
\end{equation*}
As this inequality clearly holds, at least one of $\ineqb_w(j)$ and $\ineqb_w(j+2k)$ must hold as well. As the parities of $j$ and $j+2k$ are the same, this, together with the above observation, implies that $\ineqb_w(j)$ holds for at least $2k+1$ choices of $j$ from $\{1,\dots,4k\}$ for $w\in\{w_1,w_2\}$. Therefore, there exists an index $j$ for which both $\ineqb_{w_1}(j)$ and $\ineqb_{w_2}(j)$ are satisfied. As each edge is used at most twice, the statement follows.
\end{proof}

With the help of Lemmas~\ref{lem:same}, \ref{lem:24} and \ref{lem:more}, we are ready to prove the main result of the paper.

\begin{thm}\label{thm:main1}
Let $\bP_1=(R_1,B_1)$ and $\bP_2=(R_2,B_2)$ be colorings of a wheel $G=(V,E)$, and let $w\colon E\to\mathbb{R}_+$ be a weight function. Then there exists a sequence of exchanges of length at most $n-1$ and total weight at most $w(E)$ that transforms $\bP_1$ into $\bP_2$ and uses each edge at most twice.
\end{thm}
\begin{proof}
If the colorings have identical orientation, then the theorem follows by Lemma~\ref{lem:same}. Indeed, the length and the weight of any strictly monotone sequence of exchanges that transforms $\bP_1$ into $\bP_2$ achieves the natural lower bounds $(n-1)-|R_1\cap R_2|\leq n-1$ and $w(R_1\triangle R_2)\leq w(E)$, respectively, and uses each edge at most once.

Hence assume that the colorings have different orientations. If $\bP_1$ has at most four intervals, then the theorem immediately follows by Lemma~\ref{lem:24}. Otherwise, Lemma~\ref{lem:more} with the choice $w_1\coloneqq w$ and $w_2\equiv 1$ ensures the existence of a sequence of exchanges of total weight at most $w_1(E)=w(E)$ and length at most $w_2(E)/2=|E|/2=n-1$ that uses each edge at most twice, concluding the proof.
\end{proof}

\section{Spikes}
\label{sec:spikes}

In this section, we prove that a strengthening of Conjecture~\ref{conj:weighted} analogous to Theorem~\ref{thm:main1} holds for spikes as well. That is, consider a rank-$r$ spike $M$ over ground set $S$, and let $w\colon S\to\mathbb{R}_+$ be a weight function. We show that for any two compatible basis pairs $\bP_1=(R_1,B_1)$ and $\bP_2 =(R_2, B_2)$, there exists a sequence of exchanges of length at most $r$ and total weight at most $w(S)$ that transforms $\bP_1$ into $\bP_2$ and uses each element at most twice.

Recall that $S=\{t,x_1,y_1,\dots,x_r,y_r\}$, where $t$ is the tip and $\{x_i,y_i\}$ for $1\leq i\leq r$ are the legs of the spike. Hence $R_1\cup B_1=R_2\cup B_2$ does not contain exactly one element $s$ of $S$; for short, we say that the pairs $\bP_1$ and $\bP_2$ \textbf{miss} the element $s$. We distinguish two cases depending on whether this element is the tip of $M$ or not.

\begin{lem} \label{lem:spiket}
Let $\bP_1=(R_1,B_1)$ and $\bP_2=(R_2,B_2)$ be compatible pairs of disjoint bases of a rank-$r$ spike $M$ over a ground set $S$ missing the tip $t$, and let $w\colon S\to\mathbb{R}_+$ be a weight function. Then there exists a sequence of exchanges of length at most $r$ and total weight at most $w(S-t)$ that transforms $\bP_1$ into $\bP_2$ and uses each element at most twice.
\end{lem}
\begin{proof}
Let $Z\subseteq S-t$ be a subset of size $r$. Then $Z$ is a basis of $M$ if and only if one of the followings hold.
\begin{enumerate}\itemsep0em
    \item There exist indices $k$ and $\ell$ such that $\{x_k, y_k\} \subseteq Z$, $\{x_\ell, y_\ell \} \cap Z = \emptyset$ and $|\{x_i,y_i\} \cap Z| = 1$ for all $1 \le i \le r$, $i \not \in \{k,\ell\}$.
    \item $|Z \cap \{x_i,y_i\}| = 1$ for all $1 \le i \le r$ and $Z$ is not a circuit of $M$.
\end{enumerate}
Bases of the latter type are called \textbf{transversal bases} as they intersect every leg of $M$, while bases of the former type are called \textbf{non-transversal bases}.

\medskip
\noindent \textbf{Case 1.} Both $R_1$ and $R_2$ are non-transversal.

We may assume that both $x_1$ and $y_1$ are red and both $x_r$ and $y_r$ are blue in $\bP_1$. Note that in this case exchanging the elements $x_i$ and $y_i$ is feasible for $2 \le i \le r-1$. As $R_2$ is non-transversal, there exists indices $1\leq k,\ell\leq r$ such that $x_k$ and $y_k$ are both red and $x_\ell$ and $y_\ell$ are both blue in $\bP_2$. We distinguish cases based on the values of $k$ and $\ell$.

\begin{figure}[t!]
\centering
\begin{subfigure}[t]{0.47\textwidth}
\centering
\def\svgwidth{0.85\linewidth} 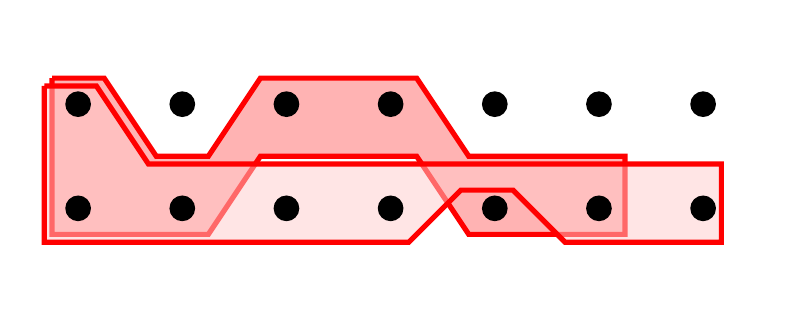
\caption{Case $k=1$.}
\label{fig:spike11}
\end{subfigure}\hfill
\begin{subfigure}[t]{0.47\textwidth}
\centering
\def\svgwidth{0.85\linewidth} 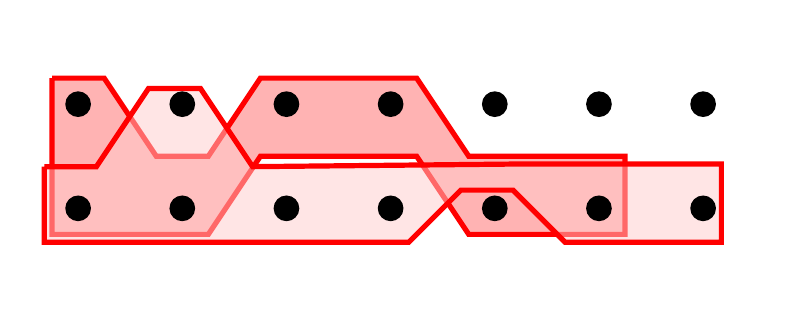
\caption{Case $2\leq k\leq r-1$.}
\label{fig:spike12}
\end{subfigure}\\[5pt]
\begin{subfigure}[t]{0.47\textwidth}
\centering
\def\svgwidth{0.85\linewidth} 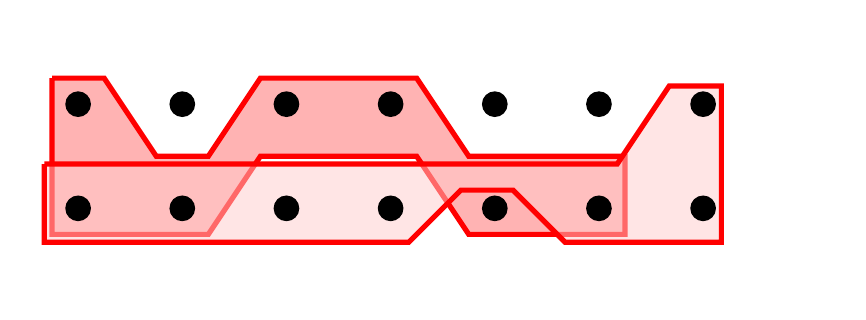
\caption{Case $k=r$ and $2\leq\ell\leq r-1$.}
\label{fig:spike13}
\end{subfigure}\hfill
\begin{subfigure}[t]{0.47\textwidth}
\centering
\def\svgwidth{0.85\linewidth} 
\begingroup%
  \makeatletter%
  \providecommand\color[2][]{%
    \errmessage{(Inkscape) Color is used for the text in Inkscape, but the package 'color.sty' is not loaded}%
    \renewcommand\color[2][]{}%
  }%
  \providecommand\transparent[1]{%
    \errmessage{(Inkscape) Transparency is used (non-zero) for the text in Inkscape, but the package 'transparent.sty' is not loaded}%
    \renewcommand\transparent[1]{}%
  }%
  \providecommand\rotatebox[2]{#2}%
  \newcommand*\fsize{\dimexpr\f@size pt\relax}%
  \newcommand*\lineheight[1]{\fontsize{\fsize}{#1\fsize}\selectfont}%
  \ifx\svgwidth\undefined%
    \setlength{\unitlength}{249.31054688bp}%
    \ifx\svgscale\undefined%
      \relax%
    \else%
      \setlength{\unitlength}{\unitlength * \real{\svgscale}}%
    \fi%
  \else%
    \setlength{\unitlength}{\svgwidth}%
  \fi%
  \global\let\svgwidth\undefined%
  \global\let\svgscale\undefined%
  \makeatother%
  \begin{picture}(1,0.36950927)%
    \lineheight{1}%
    \setlength\tabcolsep{0pt}%
    \put(0,0){\includegraphics[width=\unitlength,page=1]{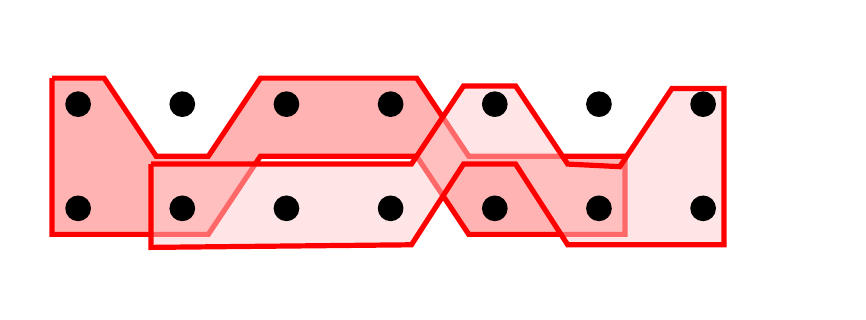}}%
    \put(-0.01,0.30934335){\makebox(0,0)[lt]{\lineheight{1.25}\smash{\begin{tabular}[t]{l}$x_1=x_\ell$\end{tabular}}}}%
    \put(0.72,0.30934335){\makebox(0,0)[lt]{\lineheight{1.25}\smash{\begin{tabular}[t]{l}$x_r=x_k$\end{tabular}}}}%
    \put(0.07,0.03859668){\makebox(0,0)[lt]{\lineheight{1.25}\smash{\begin{tabular}[t]{l}$y_1$\end{tabular}}}}%
    \put(0.55,0.03859668){\makebox(0,0)[lt]{\lineheight{1.25}\smash{\begin{tabular}[t]{l}$y_\ell$\end{tabular}}}}%
    \put(0.79,0.03859668){\makebox(0,0)[lt]{\lineheight{1.25}\smash{\begin{tabular}[t]{l}$y_r$\end{tabular}}}}%
    \put(0.36099556,0.30934335){\makebox(0,0)[lt]{\lineheight{1.25}\smash{\begin{tabular}[t]{l}$R_1$\end{tabular}}}}%
    \put(0.36099556,0.00851371){\makebox(0,0)[lt]{\lineheight{1.25}\smash{\begin{tabular}[t]{l}$R_2$\end{tabular}}}}%
  \end{picture}%
\endgroup%

\caption{Case $k=r$ and $\ell=1$.}
\label{fig:spike14}
\end{subfigure}
\caption{Illustration of Case 1 of Lemma~\ref{lem:spiket}.}
\label{fig:spike1}
\end{figure}

\smallskip
\noindent \textbf{Case 1.1.} $k=1$.

Exchange $x_i$ and $y_i$ for all $2 \le i \le r-1$ for which $i \ne \ell$ and the color of $x_i$ differs in $\bP_1$ and $\bP_2$. If $\ell \ne r$, then exchange the unique element of $\{x_\ell, y_\ell\} \cap R_1$ with the unique element of $\{x_r, y_r\} \cap R_2$. 

\smallskip
\noindent \textbf{Case 1.2.} $2 \le k \le r-1$.

We may assume that $x_1$ and $x_r$ are blue in $\bP_2$ and $x_k$ is red in $\bP_1$. Exchange $x_1$ with $y_k$, then exchange $x_i$ and $y_i$ for all $2 \le i \le r-1$ for which $i \not \in \{k, \ell\}$ and the color of $x_i$ differs in $\bP_1$ and $\bP_2$. Moreover, if $\ell \ne r$ then exchange the unique element of $\{x_\ell, y_\ell\} \cap R_1$ with $y_r$. 

\smallskip
\noindent \textbf{Case 1.3.} $k=r$ and $2 \le \ell \le r-1$.

Exchange the unique element of $\{x_\ell, y_\ell\} \cap R_1$ with $x_r$. Then exchange $x_i$ and $y_i$ for all $2 \le i \le r-1$ for which $i \ne \ell$ and the color of $x_i$ differs in $\bP_1$ and $\bP_2$. Finally, exchange the unique element of $\{x_1, y_1\} \cap B_2$ with $y_r$. 

\smallskip
\noindent \textbf{Case 1.4.} $k=r$ and $\ell=1$.

Suppose first that there exists an index $2 \le j \le r-1$ such that $x_j$ has different colors in $\bP_1$ and $\bP_2$. We may assume that $x_j$ is red in $\bP_1$.  Exchange $x_1$ with $y_j$, then exchange $x_i$ and $y_i$ for all $2 \le i \le r-1$ for which $i \ne j$ and $x_i$ has different colors in $\bP_1$ and $\bP_2$. Finally, exchange $y_1$ and $x_r$, and exchange $x_j$ with $y_r$. This results in a strictly monotone exchange sequence that transforms $\bP_1$ into $\bP_2$.  

The only remaining case is when $x_i$ has the same color in $\bP_1$ and $\bP_2$ for all $2\le i \le r-1$, we may assume that they are all red.  If $r=2$ then at least one of the exchanges $(x_1,x_2)$, $(x_1,y_2)$ is feasible, exchanging them and exchanging the other two elements afterwards gives a strictly monotone exchange sequence.
If $r \ge 3$ then we may assume that $w(y_2) \le w(x_2)$. Consider the exchange sequence $(x_1, y_2)$, $(y_1, x_r)$, $(y_2, y_r)$ of length $3 \le r$. The only element that is used twice is $y_2$, while $x_2$ was not used at all, hence the sequence has total weight at most $w(E)$ by $w(y_2) \leq w(x_2)$.

\medskip
\noindent \textbf{Case 2.} Both $R_1$ and $R_2$ are transversal bases.

If $R_1$ and $R_2$ differ in only one leg, say $R_1\cap \{x_i,y_i\}\neq R_2\cap\{x_i,y_i\}$, then exchanging $x_i$ and $y_i$ is feasible. Otherwise we may assume that $x_1,x_2 \in R_1$ and $y_1,y_2 \in R_2$. Exchange $x_1$ and $y_2$, then exchange $x_i$ and $y_i$ for all $3 \le i \le r$ for which the color of $x_i$ differs in $\bP_1$ and $\bP_2$. Finally, exchange $x_2$ and $y_1$.

\medskip
\noindent \textbf{Case 3.} Exactly one of $R_1$ and $R_2$ is a transversal basis. 

We may assume that $R_1$ is non-transversal such that both $x_1$ and $y_1$ are red and both $x_r$ and $y_r$ are blue in $\bP_1$. Exchange $x_i$ and $y_i$ for all $2 \le i \le r-1$ for which the color of $x_i$ differs in $\bP_1$ and $\bP_2$. Then exchange the unique element of $\{x_1,y_1\} \cap B_2$ with the unique element of $\{x_r,y_r\} \cap R_2$. 

\medskip
Concluding the above, in all cases except Case 1.4, the resulting sequence of exchanges is strictly monotone and uses each element at most twice, hence the bounds on the length and the total weight follow. In Case 1.4, we showed that the bounds hold, concluding the proof of the lemma. 
\end{proof}

Now we consider the case when $\bP_1$ and $\bP_2$ miss a non-tip element, say, $x_1$.

\begin{lem} \label{lem:spikex}
Let $\bP_1=(R_1,B_1)$ and $\bP_2=(R_2,B_2)$ be compatible pairs of disjoint bases of a rank-$r$ spike $M$ over a ground set $S$ missing the non-tip element $x_1$, and let $w\colon S\to\mathbb{R}_+$ be a weight function. Then there exists a sequence of exchanges of length at most $r$ and total weight at most $w(S-x_1)$ that transforms $\bP_1$ into $\bP_2$ and uses each element at most twice.
\end{lem}
\begin{proof}
Let $Z\subseteq S-x_1$ be a subset of size $r$. Then $Z$ is a basis of $M$ if and only if one of the followings hold.
\begin{enumerate}\itemsep0em
    \item $t \in Z$, $y_1 \in Z$, $\{x_\ell, y_\ell\} \cap Z = \emptyset$ for an index $2 \le \ell \le r$, $|\{x_i,y_i\} \cap Z| = 1$ for all $i \in \{2,\dots, r\}$, $i\neq\ell$.
    \item $t \in Z$, $y_1 \not \in Z$, $|\{x_i,y_i\} \cap Z| = 1$ for all $2 \le i \le r$.
    \item $t \not \in Z$, $y_1 \in Z$, $|\{x_i,y_i\} \cap Z| = 1$ for all $2 \le i \le r$, $Z$ is not a circuit.
    \item $t \not \in Z$, $y_1 \not \in Z$, $\{x_\ell, y_\ell\} \subseteq Z$ for an index $2 \le \ell \le r$, $|Z \cap \{x_i,y_i\}| = 1$ for all $i \in \{2,\dots, r\}$, $i \ne \ell$.
\end{enumerate}

We will distinguish two cases based on the coloring of $t$ and $y_1$.

\medskip
\noindent \textbf{Case 1.} $t$ and $y_1$ have the same color in at least one of $\bP_1$ and $\bP_2$.

We may assume that $t$ and $y_1$ are both red, and $x_r$ and $y_r$ are both blue in $\bP_1$. Exchange $x_i$ and $y_i$ for all $2 \le i \le r-1$ for which $|R_2 \cap \{x_i,y_i\}| = |B_2 \cap \{x_i, y_i\}| = 1$ and $x_i$ have different colors in $\bP_1$ and $\bP_2$. Let $\bP'_1=(R'_1, B'_1)$ denote the coloring obtained this way. Note that all elements have the same color in $\bP'_1$ and $\bP_2$, except possibly $t, y_1, x_r, y_r$ and at most one leg $\{x_\ell, y_\ell\}$.
We distinguish further cases based on the colors of $t$ and $y_1$ in $\bP_2$. 

\smallskip
\noindent \textbf{Case 1.1.} $t$ and $y_1$ are both red in $\bP_2$.

Let $\ell$ denote the unique index for which both $x_\ell$ and $y_\ell$ are blue in $\bP_2$. If $\ell = r$, we need no further exchanges. If $\ell \ne r$, exchange the element of the leg $\{x_\ell, y_\ell\}$ which is red in $\bP_1$ with the element of leg $\{x_r, y_r\}$ which is red in $\bP_2$. 

\smallskip
\noindent \textbf{Case 1.2.} Exactly one of $t$ and $y_1$ is red in $\bP_2$. 
In this case exactly one of $t$ and $y_1$ and exactly one of $x_r$ and $y_r$ need to change colors and their exchange is feasible. 

\smallskip
\noindent \textbf{Case 1.3} Both $t$ and $y_1$ are blue in $\bP_2$.

Let $\ell$ denote the unique index for which $x_\ell$ and $y_\ell$ are both red in $\bP_2$. 
If $\ell = r$, then consider the exchange sequences $(y_1, x_r)$, $(t, y_r)$ and $(y_1, y_r)$, $(t, x_r)$ of length two. At least one of them is feasible since at most one of $B'_1-x_r+y_1$ and $B'_1-y_r+y_1$ forms a circuit by $|(B'_1-x_r+y_1) \cap (B'_1-y_r+y_1)| = r-1$. Adding this to the exchange sequence that transformed $\bP_1$ into $\bP'_1$ we get a strictly monotone exchange sequence from $\bP_1$ to $\bP_2$. 

If $2 \le \ell \le r-1$, then we may assume that $x_\ell$ is red and $y_\ell$ is blue in $\bP_1$, while $x_r$ is red and $y_r$ is blue in $\bP_2$. If the exchange $(y_1, x_r)$ is feasible, then adding $(y_1, x_r)$, $(t, y_\ell)$ to the exchange sequence that transformed $\bP_1$ into $\bP'_1$ results in a strictly monotone exchange sequence. From now on we assume that $(y_1, x_r)$ is not feasible which means that $B'_1-x_r+y_1$ forms a circuit. As the intersection of two circuits cannot have size $r-1$, this implies that both $B'_1-y_r+y_1$ and $B'_1-\{x_r, y_\ell\}+\{x_\ell,y_1\}$ are bases.

If $w(x_\ell) \ge w(y_r)$, then extend the exchange sequence from $\bP_1$ to $\bP'_1$ by the exchanges $(y_1, y_r)$, $(y_\ell, t)$, $(x_r, y_r)$. The feasibility of these exchanges follows from $B'_1-y_r+y_1$ being a basis. The total length of the exchange sequence is at most $(r-3)+3 = r$. The sequence does not use $x_\ell$, it uses $y_r$ twice, and all the other elements at most once, so it has weight at most $w(E)$ by the assumption $w(x_\ell) \ge w(y_r)$.

If $w(x_\ell) < w(y_r)$, then extend the exchange sequence from $\bP_1$ to $\bP'_1$ by the exchanges $(x_\ell, x_r)$, $(y_1, y_\ell)$, $(x_\ell, t)$. The feasibility of these exchanges follows from $B'_1-\{x_r, y_\ell\}+\{x_\ell,y_1\}$ being a basis. The total length of the exchange sequence is at most $(r-3)+3=r$. The sequence does not use $y_r$, it uses $x_\ell$ twice, and all the other elements at most once, so it has weight at most $w(E)$ by the assumption $w(x_\ell) < w(y_r)$.

\medskip
\noindent \textbf{Case 2.} $t$ and $y_1$ have different colors in both $\bP_1$ and $\bP_2$.

We may assume that $y_1$ is red and $t$ is blue in $\bP_1$. Observe that each leg contains one red and one blue element in any of the colorings $\bP_1$ and $\bP_2$. If every leg is colored the same way in $\bP_1$ and $\bP_2$, then the exchange distance of $\bP_1$ and $\bP_2$ is at most one. Hence we may assume that $x_2 \in B_1 \cap R_2$ and $y_2 \in B_2 \cap R_1$.

\smallskip
\noindent \textbf{Case 2.1.} $y_1$ is blue and $t$ is red in $\bP_2$.

Exchange $x_2$ with $y_1$. After this step, both $y_1$ and $t$ is blue, hence we can exchange $x_i$ and $y_i$ for all $3 \le i \le r$ for which $x_i$ has different colors in $\bP_1$ and $\bP_2$. Finally, exchange $y_2$ with $t$.

\smallskip
\noindent \textbf{Case 2.2.} $y_1$ is red and $t$ is blue in $\bP_2$.

If $w(y_1) \leq w(t)$, exchange $y_1$ with $x_2$. After this step, exchange $x_i$ and $y_i$ for all $3 \le i \le r$ for which $x_i$ has different colors in $\bP_1$ and $\bP_2$. Finally, exchange $y_1$ with $y_2$. This exchange sequence does not use $t$, uses $y_1$ twice and all the other elements at most once. These imply that its length is at most $r$ and total weight is at most $w(E)$ by $w(y_1) \le w(t)$.

If $w(y_1) > w(t)$, exchange $t$ with $y_2$. After this step, exchange $x_i$ and $y_i$ for all $3 \le i \le r$ for which $x_i$ has different colors in $\bP_1$ and $\bP_2$. Finally, exchange $t$ with $x_2$. This exchange sequence does not use $y_1$, uses $t$ twice and all the other elements at most once. These imply that its length is at most $r$ and total weight is at most $w(E)$ by $w(y_1) > w(t)$.

\medskip
Concluding the above, in all cases except Cases 1.3 and 2.2, the resulting sequence of exchanges is strictly monotone and uses each element at most twice, hence the bounds on the length and the total weight follow. In Cases 1.3 and 2.2, we showed that the bounds hold, concluding the proof of the lemma.
\end{proof}

\begin{thm} \label{thm:main2}
Let $\bP_1 = (R_1,B_1)$ and $\bP_2 = (R_2,B_2)$ be compatible pairs of disjoint bases of a rank-$r$ spike $M$ over a ground set $S$, and let $w\colon  S\to \mathbb{R}_+$ be a weight function. Then there exists a sequence of exchanges of length at most $r$ and total weight at most $w(S)$ that transforms $\bP_1$ into $\bP_2$ and uses each element at most twice.
\end{thm}
\begin{proof}
    The theorem follows by combining Lemmas~\ref{lem:spiket} and~\ref{lem:spikex}.
\end{proof}

\begin{rem}
The orientation of colorings of wheels played a crucial role in the existence of strictly monotone exchange sequences. In particular, among any three basis pairs there exists two having the same orientation, and for those there exists a strictly monotone sequence of exchanges. Spikes are interesting because one can define an arbitrarily large number of colorings without a strictly monotone exchange sequence between any two of them.

To see this, consider the spike\footnote{This matroid is interesting on its own as it is isomorphic to the unique rank-$r$ binary spike, see \cite{oxley2011matroid} for details.} $M$ defined by \[\cC_3 \coloneqq \{C\subseteq S:\ |C|=r,\ |C\cap\{x_i,y_i\}|=1\ \text{for $1\leq i\leq r$},\  |C \cap \{x_1, \dots, x_r\}|\ \text{is odd}\}.\]
Consider all pairs $(R,B)$ of disjoint sets of size $r$ for which $y_1 \in R$, $t \in B$, $|R \cap \{x_i,y_i\}| = |B \cap \{x_i,y_i\}| = 1$ for all $2 \le i \le r$, and $|R \cap \{x_2, \dots, x_r\}|$ is even. There are $2^{r-2}$ such pairs. Note that all such pairs $(R,B)$ are colorings of the matroid obtained from $M$ by deleting $x_1$. Further observe that any feasible exchange uses at least one of $y_1$ and $t$. Indeed, exchanging $x_i$ with $y_j$ for some $2 \le i,j \le r$ results in an odd number of red elements in $\{x_2, \dots, x_r\}$, meaning that the red elements form a circuit, showing that the exchange is non-feasible. Exchanging $x_i$ with $x_j$ or $y_i$ with $y_j$ for some $2 \le i,j \le r$, $i \ne j$ results in either $\{t,x_i,y_i\}$ or $\{t,x_j,y_j\}$ forming a blue circuit, hence the exchange is non-feasible. This implies that there exists no strictly monotone exchange sequence between any two distinct colorings of this form.
\end{rem}

\section{Conclusions}
\label{sec:further}

In this paper, we proposed a weighted generalization of Hamidoune's conjecture on the exchange distance of compatible basis pairs. We verified the conjecture for strongly base orderable matroids, and its strengthening in which the exchange sequence had small length and weight simultaneously for split matroids, graphic matroids of wheels, and spikes. For the latter three classes, our proofs also imply polynomial-time\footnote{In matroid algorithms, it is usually assumed that the matroids are given by independence oracles, and the complexity of the algorithm is measured by the number of oracle calls and other conventional elementary steps.} algorithms that determine the required exchange sequences. For strongly base orderable matroids, we get a similar result if, for any pair of bases, a bijection $\phi$ between them ensured by strongly base orderability can be computed in polynomial time.  

Motivated by Lemma~\ref{lem:more}, it would be tempting to formulate a conjecture stating that there always exists an exchange sequence that has small weight with respect to two weight functions $w_1$ and $w_2$ simultaneously. However, this is not true in general. Let $G=(V,E)$ be a complete graph on four vertices with edge set $E=\{a,b,c,d,e,f\}$, and let $M$ be the graphic matroid of $G$, see Figure~\ref{fig:k4}. Interestingly, $G$ is a wheel, but $M$ can be also obtained by deleting the tip of a spike. Consider two colorings $\bP_1=(R_1,B_1)$ and $\bP_2=(R_2,B_2)$ of $G$ where $R_1=\{a,b,c\}$, $B_1=\{d,e,f\}$, $R_2=\{b,d,f\}$ and $B_2=\{a,c,e\}$. It is not difficult to check that any sequence of exchanges that transforms $\bP_1$ into $\bP_2$ uses at least one of the edges $b$ and $e$ twice. This implies that if we set $w_1(b)\coloneqq 1$ and $0$ otherwise, and we set $w_2(e)\coloneqq 1$ and $0$ otherwise, then any exchange sequence has $w_1$-weight or $w_2$-weight at least $2=2w_1(E)=2w_2(E)$. 

\begin{figure}[t!]
\centering
\begin{subfigure}[t]{0.47\textwidth}
\centering
\includegraphics[width=.4\linewidth]{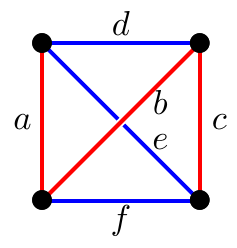}
\caption{Coloring $\bP_1=(R_1,B_1)$ of $K_4$.}
\label{fig:k4a}
\end{subfigure}\hfill
\begin{subfigure}[t]{0.47\textwidth}
\centering
\includegraphics[width=.4\linewidth]{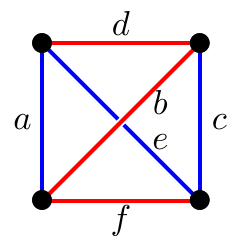}
\caption{Coloring $\bP_2=(R_2,B_2)$ of $K_4$.}
\label{fig:k4b}
\end{subfigure}
\caption{Example showing that the statement of Lemma~\ref{lem:more} is not true in general. Any exchange sequence that transforms $(R_1,B_1)$ into $(R_2,B_2)$ uses one of $b$ and $e$ twice. If $w_1(b)=1$ and $0$ otherwise, and $w_2(e)=1$ and $0$ otherwise, then any sequence violates the weight constraint for at least one the weight functions.}
\label{fig:k4}
\end{figure}

On the other hand, no counterexample is known for the case when $w_2\equiv 1$, that is, when one would like to find an exchange sequence that has small length and weight simultaneously. Even more, both Theorems~\ref{thm:main1} and~\ref{thm:main2} ensured the existence of an exchange sequence that uses each element at most twice. A rather optimistic conjecture would state the existence of a sequence of exchanges that, besides having small length and weight, also has this property. 

\paragraph{Acknowledgement.}

The authors are grateful to Attila Bernáth, Zoltán Király, Yusuke Ko\-ba\-yashi, and Eszter Szabó for helpful discussions. 

Tamás Schwarcz was supported by the \'{U}NKP-22-3 New National Excellence Program of the Ministry for Culture and Innovation from the source of the National Research, Development and Innovation Fund. The work was supported by the Lend\"ulet Programme of the Hungarian Academy of Sciences -- grant number LP2021-1/2021 and by the Hungarian National Research, Development and Innovation Office -- NKFIH, grant numbers FK128673 and TKP2020-NKA-06.


\bibliographystyle{abbrv}
\bibliography{weighted}

\end{document}